\documentclass[10pt,a4paper, reqno]{amsart} 

\usepackage{geometry}
\usepackage[T1]{fontenc}
\usepackage{lmodern}
\usepackage{fix-cm}
\usepackage[right]{eurosym} 
\usepackage[utf8]{inputenc} 
\usepackage[ngerman, english]{babel}

\usepackage[mathscr]{euscript}
\usepackage{amsmath}
\usepackage{wasysym}
\usepackage{amssymb}
\usepackage{bbm}
\usepackage{amscd}
\usepackage{esint}

\usepackage{graphicx}
\usepackage{float} 

\usepackage{color}
\usepackage{enumerate} 
\usepackage{hyperref} 
\usepackage{verbatim} 
\usepackage{tikz-cd} 

\usepackage[capitalize,nameinlink,noabbrev]{cleveref} 

\counterwithin{equation}{section}

\newtheorem{theorem}{Theorem}[section]

\newtheorem{Def}[theorem]{Definition}

\newtheorem{lemma}[theorem]{Lemma}

\newtheorem{proposition}[theorem]{Proposition}

\newtheorem{cor}[theorem]{Corollary}

\theoremstyle{definition}

\newtheorem{rem}[theorem]{Remark}

\newtheorem{bsp}[theorem]{Example}

\newcommand{\CB}{\mathcal{B}}

\newcommand{\CE}{\mathcal{E}}

\newcommand{\BR}{\mathbb{R}}

\newcommand{\SH}{\mathscr{H}}

\newcommand{\1}{\mathbbm{1}} 
\DeclareMathOperator*{\supp}{supp}

\newcommand{\norm}[1]{\left\lVert#1\right\rVert}
\newcommand{\normtwo}[1]{{\left\vert\kern-0.25ex\left\vert\kern-0.25ex\left\vert #1 
		\right\vert\kern-0.25ex\right\vert\kern-0.25ex\right\vert}}

\newcommand{\di}{\,\text{d}}
\newcommand{\dishort}{\text{d}}

\newcommand{\sk}[2]{\langle #1, #2 \rangle}

\newcommand{\floor}[1]{\lfloor #1 \rfloor}
\newcommand{\abs}[1]{\ensuremath{\left\vert#1\right\vert}}

\newcommand{\lap}{\Delta}

\newcommand{\ssubset}{\subset \joinrel \subset} 
\newcommand{\ssupset}{\supset \joinrel \supset} 
\DeclareMathOperator{\distw}{dist}
\newcommand{\dist}[2]{\distw(#1, #2)} 
\newcommand*{\diam}[1]{\text{diam}(#1)} 
\newcommand{\ie}{i.e.\,}
\newcommand{\aev}{a.e.\,}
\newcommand{\eg}{e.g.\,}

\newcommand{\icol}[1]{
	\left(\begin{smallmatrix}#1\end{smallmatrix}\right)%
}

\usepackage{csquotes}
\usepackage[
doi=false,
url=false,
backend=biber,
style=numeric,
sorting=nyvt
]{biblatex}
\usepackage[intoc]{nomencl}

\usepackage{varwidth}

\usepackage{changepage}                 
\usepackage{lipsum}

	{\end{adjustwidth}}
	{\end{adjustwidth}}

\begin{document}
	\title[Maximum principle for stable operators]{Maximum principle for stable operators}
	\author{Florian Grube, Thorben Hensiek}
	\address{Fakultät für Mathematik, Universität Bielefeld, Postfach 10 01 31, 33501 Bielefeld, Germany}
	\email{fgrube@math.uni-bielefeld.de, thensiek@math.uni-bielefeld.de}
	\makeatletter
	\@namedef{subjclassname@2020}{%
		\textup{2020} Mathematics Subject Classification}
	\makeatother
	
	
	
	\begin{abstract} We prove a weak maximum principle for nonlocal symmetric stable operators. This includes the fractional Laplacian. The main focus of this work is the regularity of the considered function. 
	\end{abstract}
	
	\maketitle
\section{Introduction}
	The study of maximum principles for harmonic functions can be traced back to the works \cite{gauss_mp} by Gauß and \cite{riemann1851grundlagen} by Riemann. They are a key tool in the theory of existence, particularly uniqueness, and regularity of solutions to linear second order elliptic equations. Let $\Omega$ be a sufficiently regular, bounded domain. For the Laplacian the following maximum principle is well known. If $u \in L^1_{\text{loc}}(\Omega)$ satisfies
	\begin{align}
		u \text{ is upper-semicontinuous in }\Omega \text{ and } u(x)&\le \tfrac{1}{\abs{\partial B(x)}}\int_{\partial B(x)} u(y) \di y \text{ for all balls }B(x)\ssubset \Omega,\label{eq:sub_mean} \\
		\limsup_{x\to z} u(x)&\le 0 \text{ for all } z\in \partial \Omega \label{eq:semicontinuous on boundary},
	\end{align}
	then $u \le 0$ in $\Omega$. The condition \eqref{eq:sub_mean} is equivalent to 
	\begin{align}
		(u, -\lap \eta)_{L^2(\Omega)}&\le 0 \text{ for all } \eta \in C_c^\infty(\Omega), \eta\ge 0,\label{eq:weaksubharmonicity}
	\end{align}
	\ie $u$ is distributional subharmonic. For both statements we refer to the book \cite[Chapter 27]{Donogue_book} by Donoghue. An alternative to the condition \eqref{eq:semicontinuous on boundary} is to assume 
	\begin{equation}\label{eq:boundaryregularity}
		\lim\limits_{\varepsilon\to 0+} \varepsilon^{-1}\mspace{-50mu}\int\limits_{\{ x\in \Omega\,|\, \dist{x}{\partial \Omega}<\varepsilon \}} \mspace{-55mu} u^+(x)\di x=0,
	\end{equation}
	see \cref{th:main_local}. Here $u^+= \max \{ u,0 \}$. The condition \eqref{eq:boundaryregularity} yields $\int_{\partial \Omega} u^+ \dishort \sigma =0$ for continuous functions $u\in C(\overline{\Omega})$, \ie $u\le 0$ on $\partial \Omega$. The goal of this paper is to prove a nonlocal version of this maximum principle. 
	
	In recent years there has been an intense study of nonlocal operators. The most prominent and well studied example is the fractional Laplacian. It is defined for $s\in(0,1)$ as
	\begin{align*}\label{eq:fractional_laplacian}
		(-\lap)^s u(x):= \text{p.v.}\int\limits_{\BR^d} \frac{u(x)-u(y)}{\abs{x-y}^{d+2s}} \di y.
	\end{align*}
By the symmetry of the fractional Laplacian, we can define the operator for functions in $L^1(\BR^d, \, (1+\abs{x})^{-d-2s} \di x )$ via
\begin{equation*}
	\langle (-\lap)^su, \eta\rangle:= (u,(-\lap)^s\eta)_{L^2(\BR^d)} \text{ for $\eta\in C_c^\infty(\BR^d)$. }
\end{equation*}	The space $L^1(\BR^d, \, (1+\abs{x})^{-d-2s} \di x )$ captures the decay of the Lévy measure $\nu(\dishort x)= \abs{x}^{-d-2\,s}\di x$ at infinity. It is typically called the tail space, \eg see the book \cite{bogdan_potential_analysis} by Bogdan et al. and Bogdan and Byczkowski \cite[Definition 3.1]{bogdan_tail}. For the fractional Laplacian our result, see \cref{Theorem 1.1}, reads as follows. If $u \in L^1(\BR^d, (1+\abs{x})^{-d-2s}\di x)$ satisfies 
	\begin{align}	
		\left(u,(-\Delta)^s \eta\right)_{L^2(\BR^d)}&\le 0 \text{ for all } \eta \in C_c^\infty(\Omega), \, \eta \ge 0,  \nonumber\\
		u&\le 0   \text{ \aev on } \Omega^c,\nonumber\\
		\lim\limits_{\varepsilon\to 0+} \varepsilon^{-s}\mspace{-50mu}\int\limits_{\{ x\in \Omega\,|\, \dist{x}{\Omega^c}<\varepsilon \} } \mspace{-55mu} &u^+(x)\di x=0, \label{eq:frac_lap_condition}
	\end{align}
	then $u\le 0$ \aev in $\Omega$. Instead of \eqref{eq:frac_lap_condition} we could assume the stronger but more accessible assumption $u^+\in L^{1} (\Omega, \dist{x}{\partial \Omega}^{-s} \di x)$. 
	
	Analogous to the case of the Laplacian, see \eqref{eq:sub_mean}, \eqref{eq:semicontinuous on boundary} and \eqref{eq:weaksubharmonicity}, Silvestre proved in \cite[Proposition 2.17]{Silvestre2007} the following weak maximum principle. Let $u\in L^1(\BR^d, \, (1+\abs{x})^{-d-2s} \di x )$ be upper-semicontinuous on $\overline{\Omega}$. If $u$ satisfies 
	\begin{align}
		(u,(-\lap)^s\eta)_{L^2(\BR^d)} &\le 0 \text{ for all nonnegative } \eta\in C_c^\infty(\Omega),\label{eq:subharmonic_fraclap_distri}\\
		u&\le 0 \text{ in } \Omega^c,\nonumber
	\end{align}
	then $u\le 0$ in $\Omega$. We want to emphasize that the function $u$ needs to be to upper-semicontinuous up to the boundary of $\Omega$. In \cite[Theorem 5.2]{cafferelli_silvestre_comparison_principle} Caffarelli and Silvestre extended this result to a larger class of operators. The condition \eqref{eq:frac_lap_condition} is less restrictive than being upper-semicontinuous on $\overline{\Omega}$ since upper-semicontinuous functions attain their maximum on compact sets.	Cabr{\'e} and Sire proved a strong maximum principle for the fractional Laplacian, using a representation as a Dirichlet-to-Neumann map, see \cite[Section 4.6]{Cabre_Sire}. In \cite{schrödinger_operator} Bogdan and Byczkowski used probabilistic methods to prove a strong maximum principle for supersolutions related to the Schrödinger operator $(-\Delta)^s + q$. Weak and strong maximum principles for a larger class of operators of the form $L +q$ can be found in the work \cite{Jarohs2019} by Jarohs and Weth. Results for solutions to nonlinear equations and antisymmetric solutions of related linear equations can be found in \cite{Jarohs2016}. For the case of the fractional Laplacian we refer the reader also to the work of Chen, Li, Li \cite{Chen2017} and Lü \cite{Lue2019}. Abatangelo proved a weak maximum principle for the fractional Laplacian in \cite[Lemma 3.9]{abatangelo_large_sol}. Remarkably, the only assumption on the function is $u\in L^1(\Omega)$. This is a consequence of allowing for test functions $\eta\in C(\BR^d)$ in \eqref{eq:subharmonic_fraclap_distri} such that $(-\lap)^s \eta = \psi$ in $\Omega$ and $\eta=0$ on $\Omega^c$. Here $\psi\in C_c^\infty(\Omega)$ is an arbitrary nonnegative function. Maximum principles and the failure thereof for higher order fractional Laplacians were discussed by Abatangelo, Jarohs and Saldaña in \cite{abatangelo_higher_order}. 
	
	A nonlocal Green-Gauß formula motivates the following bilinear form associated to the fractional Laplacian. For $u\in C_b^2(\BR^d)$ and $v\in C_c^1(\Omega)$  
	\begin{equation*}
		\int\limits_{\Omega} (-\lap)^s u(x) v(x) \di x = \tfrac{1}{2}\iint\limits_{(\Omega^c\times \Omega^c)^c}\frac{(u(x)-u(y))(v(x)-v(y))}{\abs{x-y}^{d+2\,s}}\di y \di x =:  \CE_{(-\lap)^s}(u,v). 
	\end{equation*}
	In this setup the following weak maximum principle holds. Let $u:\BR^d\to \BR$ satisfy $\CE_{(-\lap)^s}(u,u)<\infty$. If $u\le 0$ \aev on $\Omega^c$ and $u$ is a weak subsolution in $\Omega$, i.e. $\CE_{(-\lap)^s}(u,v)\le 0$ for all $v:\BR^d\to [0,\infty)$ with $\CE_{(-\lap)^s}(v,v)<\infty$ and $v=0$ \aev on $\Omega^c$, then $u\le 0$ \aev. This statement is a direct consequence of choosing $v=u^+ = \max\{ u,0 \}$, see Servadei and Valdinoci \cite[Lemma 6]{valdinoci_energy_frac_lap}. The idea is often used in the proof of maximum principles for second order elliptic operators, see Gilbarg and Trudinger \cite{gil_trud}. This approach has been applied to a larger class of nonlocal operators, including those with nonsymmetric kernels, by Felsinger, Kaßmann and Voigt in \cite[Theorem 4.1]{Felsinger2013}.
	
	In this article we consider generators of symmetric, stable Lévy processes. These processes play a key role in the Generalized Central Limit Theorem, \eg see the book \cite{zentral_limit} by Samorodnitsky and Taqqu. Let $s\in(0,1)$, $(X_t)$ a symmetric, $2s$-stable process, \ie  for all $t>0$ 
\begin{equation*}
	X_1\overset{d}{=} t^{-\tfrac{1}{2s}}X_t.
\end{equation*}
There exists  a nonnegative, finite measure $\mu$ on the unit sphere $S^{d-1}$ such that its generator is $-A_s$, where
\begin{align}
	A_s u(x):= \text{p.v.} \int\limits_{\BR^{d}} (u(x)-u(x+h)) \,\nu(\dishort h) \label{eq:definition_A_s}
\end{align}
with the Lévy measure given in polar coordinates via
\begin{align}
	 \nu(U):= \int\limits_{\BR}\int\limits_{S^{d-1}} \1_U(r\theta) \abs{r}^{-1-2s}  \mu(\dishort\theta) \di r, \quad U\in \CB(\BR^d).\label{eq:levy_measure}
\end{align}
This relation was established by Lévy in \cite{levy} and Khintchine in \cite{khintchine}. See also Sato \cite{sato} for a proof.
Additionally, we assume the nondegeneracy condition
	\begin{align}\label{eq:ellip_conditions}
		0<C\le  \inf\limits_{\omega\in S^{d-1}} \int\limits_{S^{d-1}}\abs{\omega\cdot \theta}^{2s} \mu(\dishort\theta).
	\end{align}
The condition \eqref{eq:ellip_conditions} is satisfied as soon as $\mu$ is not supported by a hyperplane. It is rooted in the work \cite{picard} of Picard as an ellipticity assumption on the Lévy measure $\nu$. Another motivation to study operators like \eqref{eq:definition_A_s} is Courr{\'e}ge's theorem, which characterizes the operators satisfying a maximum principle, see the work \cite{courrege} by Courr{\'e}ge. We emphasize two examples in this class of operators. If we pick $\mu$ as a uniform distribution on the sphere, then the resulting operator is the aforementioned fractional Laplacian $(-\lap)^s$. Another example is $\sum_{j=1}^{d} (-\partial_{j}^2)^s$, which is a sum of one dimensional fractional Laplacians in all coordinate directions. It is the generator of the process $(X_t^{1}, \dots, X_t^{d})$ where $X_t^{i}$ are independent, $1$-dimensional, symmetric, $2s$-stable processes. In this case $\mu$ is a sum of Dirac measures $\delta_{e_{i}}$, where $e_i$ are basis vectors of $\BR^d$. 

The aim of this article is to find the minimal regularity of a function $u:\BR^d \to \BR$ such that the following maximum principle holds. 
\begin{equation}\label{eq:max_principle}
	A_s u\le 0 \text{ in }\Omega, \, u\le 0 \text{ on }\Omega^c \Rightarrow u\le 0 \text{ in }\Omega.
\end{equation}
We study the operator distributionally. Thereby, 
\begin{equation}\label{eq:distributional_A_s}
	\sk{A_su}{\eta}:= (u, A_s \eta)_{L^2(\BR^d)}
\end{equation}
should exist for all $\eta \in C_c^\infty(\Omega)$. We introduce the weighted $L^1$-space $L^1(\BR^d,\, \nu^\star(x) \di x)$ with the weight
\begin{equation}\label{eq:nustar}
	\nu^\star(x):= \int\limits_{\BR}\int\limits_{S^{d-1}}  \1_{\Omega} (x+r\theta) (1+\abs{r})^{-1-2s}   \mu(\dishort\theta)\di r.
\end{equation}
We call it the tail space for $\nu$ and $\Omega$. $A_s \eta$ is bounded in $\Omega$ and decays at infinity like $\nu^\star$ for $\eta\in C_c^\infty(\Omega)$. Therefore, $u\in L^1(\Omega)\cap L^1(\BR^d,\, \nu^\star(x) \text{d} x)$ is sufficient for \eqref{eq:distributional_A_s} to exist, see \cref{prop:innerprod_exists}. In the case of the fractional Laplacian this tail space coincides with the aforementioned space $L^1(\BR^d, \, (1+\abs{x})^{-d-2s})$. This is proven in  \cref{appendix:tailspace}. In our second example $\sum_{j=1}^d (-\partial_{j}^2)^s$ the auxiliary measure $\nu^\star(x)\di x$ only measures sets close to the coordinate axes, dependent on $\Omega$. Here functions in $L^1(\BR^d,\, \nu^\star(x) \di x)$ are not necessarily integrable on $\Omega$. An example can be found in \cref{appendix:counterexample}. The weight $\nu^\star$ captures the behavior of $\nu$ at infinity. Foghem and Kaßmann introduced and discussed several possibilities of tail weights for a large class of Lévy measures in \cite{kassmann_foghem_neumann}.

The integrability $u\in L^1(\Omega)\cap L^1(\BR^d,\, \nu^\star(x) \di x)$ is not sufficient for \eqref{eq:max_principle} to hold. The function 
\begin{equation}\label{eq:counter_example}
	u(x):= \big(1-\abs{x}^2\big)^{-1+s} \, \text{ for } x\in B_1(0),\qquad u(x):= 0 \,\text{ for } x\notin B_1(0)
\end{equation}
satisfies $(-\lap)^s u =0$ in $B_1(0)$ and $u=0$ on $B_1(0)^c$, but it disobeys the maximum principle, see the works \cite{counterexample_hmissi} by Hmissi, \cite{counterexample_bogdan} by Bogdan and \cite{Dyda2021} by Dyda for a proof.
Now we state the main result of this article.
\begin{theorem}\label{Theorem 1.1}
	Let $\Omega\subset \BR^d$ be a bounded Lipschitz domain satisfying uniform exterior ball condition, $s\in(0,1)$ and $\mu$ a nonnegative, finite measure on the unit sphere satisfying the nondegeneracy assumption \eqref{eq:ellip_conditions}, $A_s$ be as in \eqref{eq:definition_A_s}. If $u \in L^1(\Omega) \cap L^1(\BR^d, \nu^\star(x)\di x)$ satisfies
	\begin{align}
		\left(u,A_s \eta\right)_{L^2(\BR^d)}&\le 0 \text{ for all } \eta \in C_c^\infty(\Omega), \, \eta \ge 0, \label{eq:ultrasubharmonic} \\
		u&\le 0   \text{ \aev on } \Omega^c, \label{eq:nonnegativeoutside}\\
		\lim\limits_{\varepsilon\to 0+} \varepsilon^{-s}\mspace{-50mu}\int\limits_{\{x\in \Omega\,|\, \dist{x}{\Omega^c}<\varepsilon\}}\mspace{-55mu}&u^+(x)\di x=0,\label{eq:boundary_regularity_assumption}
	\end{align}
	 then $u\le 0$ \aev in $\Omega$.
\end{theorem} 
\begin{rem}\label{rem:theorem}
	Instead of \eqref{eq:boundary_regularity_assumption}, we can assume the stronger but more accessible condition 
	\begin{equation*}
		u^+\in L^{1} (\Omega, \dist{x}{\Omega^c}^{-s} \di x) .
	\end{equation*}
	This condition implies \eqref{eq:boundary_regularity_assumption} because if $u^+\in L^{1} (\Omega, \dist{x}{ \Omega^c}^{-s} \di x)$, then
	\begin{align*}
		 \varepsilon^{-s}\mspace{-50mu}\int\limits_{\{x\in \Omega\,|\, \dist{x}{\Omega^c}<\varepsilon\}}\mspace{-50mu}u^+(x)\di x\le \int\limits_{\{x\in \Omega\,|\, \dist{x}{\Omega^c}<\varepsilon\}}\mspace{-50mu}u^+(x) \dist{x}{\Omega^c}^{-s}\di x\to 0 \text{ as } \varepsilon\to 0+.
	\end{align*}
\end{rem}
\begin{rem}
	\begin{enumerate}[(i)]
		\item{ By \eqref{eq:A_seta_tail_estimate}, if $u$ satisfies \eqref{eq:ultrasubharmonic}, then $u\1_{\supp(\nu^\star)}$ satisfies \eqref{eq:ultrasubharmonic}. Thus, we may replace \eqref{eq:nonnegativeoutside} with $u\le 0   \text{ \aev on } \Omega^c \cap \supp(\nu^\star)$ in \cref{Theorem 1.1}. This is particularly interesting if the Lévy measure $\nu$ does not have full support.}
		\item{\cref{Theorem 1.1} is optimal in the following sense. The function $u$ from \eqref{eq:counter_example} disobeys the maximum principle and	\begin{equation*}
			\varepsilon^{-s+\delta}\int\limits_{B_1\setminus B_{1-\varepsilon}} (1-\abs{x}^2)^{-1+s} \di x \to 0 \text{ as } \varepsilon\to 0+
		\end{equation*}
		for every $\delta>0$ but not for $\delta=0$.  }
		\item{ For the fractional Laplacian and the domain $\Omega=B_1(0)$, Lü proved a maximum principle in \cite[Theorem 6]{Lue2019}. The proof relies on the explicit representation of the Poisson kernel. \cite[Theorem 6]{Lue2019} requires $u\in L^{1/(1-s)}(B_1(0))\cap L^1(\BR^d, (1+\abs{x})^{-d-2s}  \dishort x)$. The assumptions on $u$ in \cref{Theorem 1.1} are weaker. If $u\in L^{1/(1-s)}(B_1(0))  $, then Hölder inequality yields \begin{equation*}
			\varepsilon^{-s}\mspace{-15mu}\int\limits_{B_1\setminus B_{1-\varepsilon}} \mspace{-20mu}u^+(x) \di x\le \varepsilon^{-s}\norm{u}_{L^{\tfrac{1}{1-s}}(B_1\setminus B_{1-\varepsilon})}\,\Big( \abs{B_1\setminus B_{1-\varepsilon}} \Big)^{s}\le c\, \norm{u}_{L^{\tfrac{1}{1-s}}(B_1\setminus B_{1-\varepsilon})}\to 0 \text{ as }\varepsilon\to 0.
		\end{equation*}}
		\item{ Very recently the articles \cite{Li_Li_maximum, Li_Li_maximum_2} by Li and Liu have been uploaded to arXiv. \cite[Theorem 1.1]{Li_Li_maximum} contains our \cref{Theorem 1.1} in the special case of the fractional Laplacian. They assume the same conditions on the function $u$. Note that the proof is rather different from ours. In contrast, our result holds for the larger class of stable, nondegenerate operators. }
	\end{enumerate}
\end{rem}
	Let us explain the main ideas in the proof of \cref{Theorem 1.1}. In an ideal situation we would use the Green function in place of $\eta$ in \eqref{eq:ultrasubharmonic}. This is not permitted in our setup. Thereby, we approximate the Green function. We consider a sequence $D_\varepsilon \ssubset \Omega$ of subdomains exhausting $\Omega$. We fix $\psi\in C_c^\infty(\Omega)$ and solve the Dirichlet problem $A_s \phi_\varepsilon = \psi$ in $D_\varepsilon$ and $\phi_\varepsilon = 0$ on $D_\varepsilon^c$. The solutions $\phi_\varepsilon$ approximate the Green function if we pick an approximate identity in place of $\psi$. Regularity results on the solution $\phi_\varepsilon$ are crucial in our proof.
	\begin{rem}
		Assumptions on the regularity of the boundary of the domain $\Omega$ in \cref{Theorem 1.1} seem unnatural. They are only needed for boundary regularity of solutions to the Dirichlet problem in \cref{prop:regularity}, see Ros-Oton and Serra \cite[Proposition 4.5]{RosOtonOuterRegularity}.
	\end{rem}
	Nonlocal operators and related Dirichlet problems are often studied in a Hilbert space setting. Following the works \cite{Felsinger2013} by Felsinger, Kaßmann, Voigt and \cite{valdinoci_energy_frac_lap} by Servadei and Valdinoci, we define a bilinear form associated to $A_s$ 
	\begin{equation*}
		\CE_{A_s}(u,v):= \tfrac{1}{2} \int\limits_{\BR^d} \int\limits_{\BR} \int\limits_{S^{d-1}}  \1_{(\Omega^c\times\Omega^c)^c}(x, x+r\,\theta) \frac{(u(x)-u(x+r\,\theta))(v(x)-v(x+r\,\theta))  }{\abs{r}^{1+2\,s}} \mu(\dishort\theta)  \di r\di x.
	\end{equation*}
	This is motivated by a nonlocal Green-Gauß formula, see Du et al. \cite{vector_calc} for bounded kernels, Dipierro, Ros-Oton and Vladinoci \cite[Lemma 3.3]{valdinoci_neumann_frac_lap} for the fractional Laplacian and Foghem and Kaßmann \cite{kassmann_foghem_neumann} for more general Lévy measures.  
	\begin{cor}\label{Corollary 1.4}
		Let $\Omega$ and $\mu$ satisfy the assumptions from \cref{Theorem 1.1} and $A_s$ be as in \eqref{eq:definition_A_s}. If $u\in H^s(\Omega)\cap L^1(\BR^d, \nu^\star(x)\dishort x)$ satisfies 
		\begin{align}
		\CE_{A_s}(u,\eta)&\le 0 \text{ for all } \eta \in C_c^\infty(\Omega), \eta \ge 0,\label{eq:weakmaximum_cond1}\\
		u&\le 0 \text{ \aev on } \Omega^c, \label{eq:weakmaximum_cond2}
		\end{align}
		 then $u\le 0$ \aev in $\Omega$. 
	\end{cor}
	\begin{rem}
		\begin{enumerate}[(i)]
			\item{ The assumption $u\in H^s(\Omega)\cap L^1(\BR^d, \nu^\star(x)\dishort x)$ implies that $\CE_{A_s}(u,\eta)$ exists, see \cref{lem:energy_exists}. }
			\item{ The proof of \cref{Corollary 1.4} uses $u\in H^{s/2}(\Omega)$ and fractional Hardy inequality to deduce the integrability $u\in L^1(\Omega, \dist{x}{\partial\Omega}^{-s}\dishort x)$. But $u\in H^{s/2}(\Omega)\cap L^1(\BR^d, \nu^\star(x)\dishort x)$ is not sufficient for $\CE_{A_s}(u,\eta)$ to exist for all $s\in (0,1)$. }
		\end{enumerate}
	\end{rem}
\subsection{Outline}
In \cref{Preliminaries} we prove basic properties of the operator $A_s$ and introduce the weak solution concept. Additionally, we state regularity results of solutions and provide technical ingredients for the proof of \cref{Theorem 1.1}. \cref{prop:calculation_distancefunction} connects the Hölder regularity of solutions with the assumption \eqref{eq:boundary_regularity_assumption}.
In \cref{sec:proof_theorem} we prove \cref{Theorem 1.1} and \cref{Corollary 1.4}.
In \cref{sec:appendix} we compare tail spaces and in \cref{appendix B} we prove a maximum principle for the Laplacian.

\subsection*{Acknowledgments}
Financial support by the German Research Foundation (GRK 2235 - 282638148) is gratefully acknowledged. We would like to thank Moritz Kaßmann and Tobias Weth for very helpful discussions.
\section{Preliminaries}\label{Preliminaries}
We set $a\wedge b := \min\{ a,b \}$, $a \vee b := \max \{ a,b \}$ and $a^+:=\max\{ a,0 \}$ for $a,b\in \BR$. We denote the set of Hölder continuous functions on $\Omega \subset \BR^d$ by $C^\alpha(\Omega) = C^{\floor{\alpha},\alpha-\floor{\alpha}}(\Omega)$ for all $\alpha>0$. In case that $\Omega$ is bounded, we equip the space $C^{\alpha}(\overline{\Omega})$ with the usual norm 
\begin{equation*}
		\norm{u}_{C^\alpha(\Omega)}:= \sum\limits_{\abs{\beta}\le \floor{\alpha}} \sup\limits_{x\in \Omega} \abs{\partial^\beta u(x)} + \sum\limits_{\abs{\beta}=\floor{\alpha}} [\partial^{\beta} u]_{C^{\alpha-\floor{\alpha}}(\Omega)     }
	\end{equation*}
	with the Hölder semi norm for $0< s <1$ \begin{equation*}
		[u]_{C^s(\Omega)}= \sup\limits_{x,y\in \Omega} \frac{\abs{u(x)-u(y)}}{\abs{x-y}^s}.
	\end{equation*} $H^s(\Omega)$ is the standard $L^2$-based Sobolev-Slobodeckij space for $s\in(0,1)$. We also define the distance function 
	\begin{align*}
		\dist{x}{\Omega} := \inf\{ \abs{x-y}\,|\, y\in \Omega \},\qquad
		\dist{\Omega}{\Omega'} := \inf \{ \dist{x}{\Omega'}\,|\, x\in \Omega \},
	\end{align*}
	where $\Omega,\Omega'\subset \BR^d$ are open sets and $x\in \BR^d$. For an open set $\Omega\subset\BR^d$ and $\varepsilon>0$ we introduce the thinned and thickened sets
	\begin{equation}\label{eq:definition_omega_epsilon}
		\Omega_{\varepsilon}:=\{ x\in \Omega\,|\, \dist{x}{\partial\Omega}>\varepsilon\},\qquad
		\Omega^{\varepsilon}:=\{ x\in \BR^d\,|\, \dist{x}{\Omega}<\varepsilon\}.
	\end{equation}
	We say that a domain $\Omega\subset\BR^d$ satisfies uniform exterior ball condition if there exists a radius $\rho>0$ such that for every $x\in \partial\Omega$ there exists a ball $B\subset \overline{\Omega}^c$ of radius $\rho$ such that $\overline{B}\cap \overline{\Omega}= \{x\}$. \smallskip
	
	The following lemma ensures the existence of a sequence of $C^\infty$ subdomains exhausting a bounded Lipschitz domain satisfying uniform exterior ball condition. This type of domain exhaustion is classical. We are particularly interested in a uniform bound of the exterior ball radius. The result is taken from the work by Mitrea, see \cite[Lemma 6.4]{mitrea}.
	\begin{lemma}{\cite[Lemma 6.4]{mitrea}}\label{kor:approximation_radius}
		Let $\Omega\subset \BR^d$ be a bounded Lipschitz domain satisfying uniform exterior ball condition. There exists a sequence of $C^\infty$-domains $\{ D_\varepsilon \}_{0<\varepsilon<\varepsilon_0}$ such that
		\begin{enumerate}[(i)]
			\item{ $D_\varepsilon\subset D_{\varepsilon'}\ssubset \Omega$ for all $0<\varepsilon'<\varepsilon<\varepsilon_0$, $\cup_{0<\varepsilon<\varepsilon_0}D_\varepsilon=\Omega$, }
			\item{ $D_\varepsilon$ satisfies uniform exterior ball condition with a radius independent of $\varepsilon$, }
			\item{ there exists a constant $\lambda>1$ such that 
				\begin{equation*}
					\varepsilon\le \dist{\partial D_\varepsilon}{\partial \Omega}, \quad \sup\{ \dist{x}{\partial \Omega}\,|\, x\in \partial D_\varepsilon \}\le \lambda \,\varepsilon
				\end{equation*}		
				for all $0<\varepsilon<\varepsilon_0$.
			}
		\end{enumerate}
	\end{lemma}   
	\begin{proof}
		The existence of the sequence $\{D_\varepsilon\}_{\varepsilon>0}$ satisfying the properties (i) and (ii) follow from \cite[Lemma 6.4]{mitrea}. Property (iii) can be ensured by choosing the constants in the construction of $D_\varepsilon$ in the proof of \cite[Lemma 6.4]{mitrea} accordingly.
	\end{proof}
	The symmetry of the Lévy measure $\nu$ allows us to rewrite $A_s$ to a double difference. 
	\begin{proposition}\label{prop:double_difference}
		Fix an open set $\Omega\subset \BR^d$ and $\alpha>0$. Let $u\in C(\BR^d) \cap C^{2s+\alpha}(\Omega)\cap L^1(\BR^d,\nu^\star(x)\dishort x)$. The term $A_s u(x)$ exists and 
		\begin{equation*}
			A_s u(x)= \frac{1}{2}\int\limits_{ \BR }\int\limits_{S^{d-1}} \frac{2\,u(x)-u(x-r\,\theta)-u(x+r\,\theta)}{\abs{r}^{1+2\,s}}  \mu(\dishort\theta) \di r
		\end{equation*}
		for any $x \in \Omega$.	Furthermore, the function $A_su $ is continuous in $\Omega$.
	\end{proposition}
This result is standard, \eg see Bogdan Byczkowski \cite[Lemma 3.5]{bogdan_tail}, Silvestre \cite[Proposition 2.4]{Silvestre2007} and Di Nezza, Palatucci and Valdinoci \cite[Lemma 3.2]{hitchhiker}.
	\begin{proof}
		Without loss of generality. $2\,s+\alpha<2$ and $2\,s+\alpha<1$ for $s<\tfrac{1}{2}$. Take $x\in \Omega$ and any ball $B_\delta(x) \ssubset \Omega$. We define for $\kappa>0$ 
		\begin{equation}\label{eq:approx_operator}
			A_s^\kappa u(x):= \int\limits_{\{\abs{r}\ge \kappa  \}}\int\limits_{S^{d-1}} \frac{u(x)-u(x+r\,\theta)}{\abs{r}^{1+2\,s}} \mu(\dishort\theta) \di r.
		\end{equation}
		This term exists by the assumption $u\in L^1(\BR^d,\nu^\star(x)\dishort x)$. Now we write $A_s^{\kappa}u(x) = \tfrac{1}{2}A_s^{\kappa}u(x)+\tfrac{1}{2}A_s^{\kappa}u(x)$ and use the transformation $r\mapsto -r$ in its second occurrence.
		\begin{align}
		A_s^\kappa u(x)&= \frac{1}{2}\int\limits_{\{\abs{r}\ge \kappa  \}}\int\limits_{S^{d-1}} \frac{u(x)-u(x+r\,\theta)}{\abs{r}^{1+2\,s}} \mu(\dishort\theta) \di r + \frac{1}{2}\int\limits_{\{\abs{r}\ge \kappa  \}}\int\limits_{S^{d-1}} \frac{u(x)-u(x-r\,\theta)}{\abs{r}^{1+2\,s}} \mu(\dishort\theta) \di r\nonumber\\
		&= \frac{1}{2}\int\limits_{\{\abs{r}\ge \kappa  \}}\int\limits_{S^{d-1}} \frac{2\,u(x)-u(x-r\,\theta)-u(x+r\,\theta)}{\abs{r}^{1+2\,s}} \mu(\dishort\theta) \di r.\label{eq:doubledifference}
		\end{align}
		For $s<\tfrac{1}{2}$, any $\theta\in S^{d-1}$ and $\abs{r}<\delta$
		\begin{equation}\label{eq:double_diff_estimate_1}
			\frac{\abs{2\,u(x)-u(x-r\,\theta)-u(x+r\,\theta)}}{2\abs{r}^{1+2\,s}}\le
				[u]_{C^{2s+\alpha}(B_\delta(x))} \abs{r}^{\alpha-1}.
		\end{equation}
		For $s\ge \tfrac{1}{2}$, $\theta\in S^{d-1}$ and $\abs{r}<\delta$ the fundamental theorem of calculus yields 
		\begin{align}\label{eq:double_diff_estimate_2}
			\abs{2\, u(x)-u(x-r\theta)-u(x+r\theta)}&=  \big| \int_0^1 \big(\nabla u(x-t\,r\theta)-\nabla u(x+t\,r\theta)\big)\cdot r\theta \di t \big|\nonumber\\
			&\le  \frac{(2\abs{r})^{2\,s+\alpha}}{2\,s+\alpha}\,[u]_{C^{2s+\alpha}(B_\delta(x))} .
		\end{align}
		Lastly, for all $s\in(0,1)$, $\theta \in S^{d-1}$ and $\abs{r}\ge \delta$
		\begin{equation}\label{eq:double_diff_estimate_3}
			\frac{\abs{2\,u(x)-u(x-r\,\theta)-u(x+r\,\theta)}}{2\abs{r}^{1+2\,s}}\le\big(1+\tfrac{1}{\delta}\big)^{1+2s} \frac{\abs{u(x)}+\abs{u(x+r\theta)}+\abs{u(x+r\theta)}}{(1+\abs{r})^{1+2s}}.
		\end{equation}
		Additionally, we know using transformation theorem and Fubini's theorem that $u\in L^1(\BR^d, \nu^\star(x)\dishort x)$ if
		\begin{equation*}
			\int\limits_{\Omega}\int\limits_{\BR} \int\limits_{S^{d-1}}  \abs{u(y+r\theta)} (1+\abs{r})^{-1-2s}  \mu(\dishort\theta) \di r \di y <\infty.
		\end{equation*}
		Therefore, the right-hand side of \eqref{eq:doubledifference} is finite for \aev $x\in \Omega$. Since $u$ is continuous on $\BR^d$, this property holds for every $x\in \Omega$. By the dominated convergence theorem with the previous bounds, the limit $\kappa\to 0$ in \eqref{eq:doubledifference} exists. Lastly, we show that $A_s u$ is continuous in $\Omega$. We fix $x\in \Omega$ and write $A_s u(x)$ as 
		\begin{equation*}
			A_su(x)= \frac{1}{2}\Bigg(\int\limits_{\{\abs{r}\ge \delta  \}}+\int\limits_{\{\abs{r}< \delta  \}}\Bigg)\int\limits_{S^{d-1}} \frac{2\,u(x)-u(x-r\,\theta)-u(x+r\,\theta)}{\abs{r}^{1+2\,s}} \mu(\dishort\theta) \di r.
		\end{equation*}
		The dominated convergence theorem, the estimates \eqref{eq:double_diff_estimate_1}, \eqref{eq:double_diff_estimate_2} and \eqref{eq:double_diff_estimate_3} and the continuity of $u$ yield the continuity of $A_s u$ at $x$.
	\end{proof}
	\begin{proposition}\label{prop:Linfty_bound}
		Let $\Omega\subset \BR^d$ be open and bounded, $s\in(0,1)$ and $\alpha>0$. There exists a constant $C>0$ such that
		\begin{equation}\label{eq:fraclaphoelderinequality}
			\norm{A_s \phi}_{L^\infty(\Omega)}\le C\, \norm{\phi}_{C^{2s+\alpha}(\Omega)}
		\end{equation}
		for all  $\phi \in C_c^{2s+\alpha}(\Omega)$ extended by zero to $\BR^d$. 
	\end{proposition}
	\begin{proof}
		The claim follows by using the arguments in \cref{prop:double_difference} uniformly for any $x\in \Omega$. Instead of picking a small ball $B_\delta(x)$ in the beginning of the proof of \cref{prop:double_difference}, we choose a ball $B\ssupset \Omega$.
	\end{proof}
	\begin{lemma}\label{prop:innerprod_exists}
		Let $\Omega\subset \BR^d$ be open and bounded. If $u\in L^1(\Omega)\cap L^1(\BR^d, \nu^\star(x)\dishort x)$ and $\eta\in C_c^\infty(\Omega)$, then $(u,A_s \eta)_{L^2(\BR^d)} $ exists.
	\end{lemma}
	\begin{proof}
		\cref{prop:Linfty_bound} yields $A_s\eta \in L^\infty(\Omega)$. This and $u\in L^1(\Omega)$ imply that $(u,A_s\eta)_{L^2(\Omega)}$ exists. Fix $\delta= \dist{\supp \eta}{\Omega^c}>0$. For any $x\in \Omega^c$
		\begin{align*}
			A_s\eta(x)= -\tfrac{1}{2}\int\limits_{\{ \abs{r}>\delta \}}\int\limits_{S^{d-1}} \frac{\eta(x-r\theta)+ \eta(x+r\theta)}{\abs{r}^{1+2s}}\mu(\dishort \theta) \di r= -\int\limits_{\{ \abs{r}>\delta \}}\int\limits_{S^{d-1}} \frac{\eta(x+r\theta)}{\abs{r}^{1+2s}}\1_{\Omega}(x+r\theta)\mu(\dishort \theta) \di r.
		\end{align*}
		Recall the definition of $\nu^\star$ in \eqref{eq:nustar}. Thereby, for $x\in \Omega^c$
		\begin{equation}\label{eq:A_seta_tail_estimate}
			\abs{A_s\eta(x)}\le \norm{\eta}_{L^\infty}\big(1+\tfrac{1}{\delta}\big)^{1+2s} \nu^\star(x).
		\end{equation}
		The claim follows from $u\in L^1(\BR^d,\nu^\star(x)\dishort x)$.
	\end{proof}
	\begin{lemma}\label{lem:energy_exists}
		Let $\Omega$ be a bounded Lipschitz domain. For any $u \in H^s(\Omega)\cap L^1(\BR^d, \nu^\star(x)\dishort x)$ and $\eta \in C_c^\infty(\Omega)$ the bilinear form $\CE_{A_s}(u,\eta)$ exists.
	\end{lemma}
	\begin{proof}
		We split the energy into $\Omega\times \Omega, \Omega\times \Omega^c$ and $\Omega^c \times \Omega$. The proof for $\Omega\times \Omega^c$ and $\Omega^c \times \Omega$ coincide. We begin with $\Omega^c\times \Omega$. Define $\delta:= \dist{\supp \eta}{\Omega^c}>0$, $C:= (1+1/\delta)^{1+2s}$ and notice $C\abs{r}^{1+2s}\ge (1+\abs{r})^{1+2\,s}$ for any $r>\delta$. Thus,
		\begin{align*}
			&\Big| \int\limits_{\Omega^c} \int\limits_{\BR}\int\limits_{S^{d-1}} \1_{\Omega}(x+r\,\theta)\frac{(u(x)-u(x+r\,\theta))(0-\eta(x+r\theta)  )}{\abs{r}^{1+2\,s}} \mu(\dishort\theta) \di r \di x\Big|\\
			&\qquad\le C \, \norm{\eta}_{L^\infty} \, \int\limits_{\Omega^c} \int\limits_{\{ \abs{r}>\delta \}}\int\limits_{S^{d-1}} \1_{\Omega}(x+r\,\theta)\frac{\abs{u(x)-u(x+r\,\theta)}}{(1+\abs{r})^{1+2\,s}} \mu(\dishort\theta) \di r \di x\\
			&\qquad\le C\, \norm{\eta}_{L^\infty}\, \Big(\norm{u}_{L^1_{\nu^\star}(\BR^d)}+ \tfrac{1}{s\, \delta^{2\,s}} \norm{u}_{L^1(\Omega)} \Big).
		\end{align*}
		For $\Omega\times \Omega$ we want to apply \cite[Proposition 6.1]{Kassmann_Dyda_Regularity} by Dyda and Kaßmann. This is possible since the Lévy measure satisfies 
		\begin{equation*}
			\int\limits_{\BR^d} \big( t \wedge \abs{z} \big)^2 \nu(\dishort z)= \mu(S^{{d-1}}) t^{2-2s}\big( \frac{1}{1-s}+ \frac{1}{s} \big)
		\end{equation*}
		for all $t>0$ and we assumed $\Omega$ to be bounded with Lipschitz boundary. Therefore, \cite[Proposition 6.1]{Kassmann_Dyda_Regularity} and Hölder inequality yield
		\begin{equation*}
			\Big| \int\limits_{\Omega} \int\limits_{\BR^d} (u(x)-u(x+h))(\eta(x)-\eta(x+h)  )\1_{\Omega}(x+h) \nu(\dishort h)  \di x\Big|\le c\, \norm{u}_{H^s(\Omega)}\,\norm{\eta}_{H^s(\Omega)}.
		\end{equation*}
	\end{proof}	
	\begin{Def}[Weak solution]\label{def:weaksol}
		Let $\Omega\subset \BR^d$ be a bounded domain and $\psi\in L^{\infty}(\Omega)$. We say that $\phi$ is a weak solution to the problem
		\begin{alignat}{2}
		A_s \phi &= \psi &\text{ in } \Omega,\nonumber\\
		\phi &= 0 &\text{ on } \Omega^c,\label{eq:fracequation}
		\end{alignat} 
		if $\phi\in L_{\text{loc}}^{1}(\BR^d)$, $\phi=0$ on $\Omega^c$ and 
		\begin{equation*}
			\int\limits_{\BR^d} \phi(x) A_s\eta(x) \di x = \int\limits_{\Omega} \psi(x) \eta(x)\di x
		\end{equation*}
		for all $\eta \in C_c^\infty(\Omega)$.
	\end{Def}	
	There is a rich theory on existence and uniqueness of weak solutions. We refer the reader to Bogdan et al. \cite{bogdan_trace}, Felsinger, Kaßmann and Voigt \cite{Felsinger2013}, Grzywny, Kaßmann and Le\.{z}aj \cite{pointwisesolutionkassmann} and Rutkowski \cite{existence_solutions_rutkowski}. In the following proposition we use the existence theorem from \cite{existence_solutions_rutkowski} to deduce the existence of solutions in the sense of \cref{def:weaksol}. 
	\begin{proposition}\label{th:pointwisesolution}
		Let $\Omega\subset \BR^d$ be an open and bounded set and $\psi \in L^\infty(\Omega)$. There exists a solution $\phi \in L^\infty(\Omega)$ to \eqref{eq:fracequation}.
	\end{proposition}
	\begin{proof}
		We introduce the function space $V_{\nu, 0}(\Omega\,|\, \BR^d)= \{ u:\BR^d \to \BR  \,|\, \CE_{A_s}(u,u)<\infty, u =0 \text{ on } \Omega^c \}$ endowed with the inner product $(u,v)_{V_{\nu}}:= (u,v)_{L^2(\Omega)}+ \CE_{A_s}(u,v)$. In \cite[Theorem 4.2]{existence_solutions_rutkowski} Rutkowski proved, via an application of the Lax-Milgram lemma, the existence of an unique $u\in V_{\nu, 0}(\Omega\,|\, \BR^d)$ satisfying
		\begin{equation}\label{eq:weak_sol}
			\CE_{A_s}(u,v)= \int_{\Omega} \psi(x)v(x) \di x 
		\end{equation} 
	for all $v\in V_{\nu,0}(\Omega\,|\,\BR^d)$.
	\cite[Lemma 5.7]{existence_solutions_rutkowski} implies $u \in L^{\infty}(\Omega)$. It remains to show that $u$ is a solution in the sense of \cref{def:weaksol}. Take any $\eta \in C_c^\infty(\Omega)\subset V_{\nu,0}(\Omega\,|\, \BR^d)$. A small calculation yields
	\begin{equation}\label{eq:greengauß_kappa}
		\int\limits_{\BR^d} u(x) A_s^\kappa \eta(x)\di x =  \tfrac{1}{2}\int\limits_{\BR^d} \int\limits_{\{\abs{r}>\kappa\}} \int\limits_{S^{d-1}}  \frac{(u(x)-u(x+r\,\theta))(\eta(x)-\eta(x+r\,\theta))  }{\abs{r}^{1+2\,s}}\mu(\dishort\theta)  \di r\di x.
	\end{equation}
	The left-hand side of \eqref{eq:greengauß_kappa} converges to $\int_{\Omega} u(x) A_s\eta(x) \di x$ by dominated convergence since $u\in L^1(\Omega)$, $\supp(u)\subset \overline{\Omega}$ and $A_s \eta\in L^\infty(\Omega)$ by \eqref{eq:fraclaphoelderinequality}. The right-hand side of \eqref{eq:greengauß_kappa} converges to $\CE_{A_s}(u,\eta)$ by dominated convergence. Thus, the solution $u$ satisfies $\int u A_s\eta = \int \psi \eta$ for all $\eta \in C_c^\infty(\Omega)$.   
	\end{proof}
	Hölder regularity up to the boundary of solutions to Dirichlet problems was proven by Ros-Oton and Serra in \cite{RosOtonOuterRegularity}. The following proposition is taken from their work. For the fractional Laplacian see also Ros-Oton and Serra \cite{ros_oton_frac_lap}.
	\begin{proposition}{\cite[Proposition 4.5]{RosOtonOuterRegularity}}\label{prop:regularity}
		Let $\Omega$ be a bounded Lipschitz domain satisfying the uniform exterior ball condition. Suppose $s\in (0,1)$ and let $\mu$ be any bounded measure on $S^{d-1}$ satisfying the nondegeneracy assumptions \eqref{eq:ellip_conditions}. Let $\psi \in L^\infty(\Omega)$ and $\phi$ be a bounded weak solution of \eqref{eq:fracequation}. Then $\phi\in C^s(\overline{\Omega})$ and there exists a constant $C_{d,\Omega, s, \mu}>0$ such that 
		\begin{equation*}
			\norm{\phi}_{C^s(\Omega)}\le C_{d,\Omega, s, \mu} \, \norm{\psi}_{L^\infty(\Omega)}.
		\end{equation*}
	\end{proposition}
	\begin{rem}\label{rem:constant}
		The constant $C_{d,\Omega, s, \mu} $ depends on $d,s$, the nondegeneracy constant of $\mu$, see \eqref{eq:ellip_conditions}, the diameter of $\Omega$, the Lipschitz constant and the uniform exterior ball radius of $\Omega$. See \cite[Proposition 4.5]{RosOtonOuterRegularity} and \cite[Proposition 1.1, Lemma 2.7, Lemma 2.9]{ros_oton_frac_lap}.
	\end{rem}
	Additional regularity inside of the domain was proven in the same work \cite[Theorem 1.1]{RosOtonOuterRegularity}. We reduce their result to fit our purposes. 
	\begin{theorem}{\cite[Theorem 1.1]{RosOtonOuterRegularity}}\label{th:interior_regularity}
		Let $s\in (0,1)$ and $\psi\in C^{\alpha}(\overline{B_1})$ for some $\alpha>0$ and $\phi$ be any bounded weak solution to 
		$$
			A_s \phi = \psi \text{ in } B_1.
		$$
		If $\phi \in C^{\alpha}(\BR^d)$, then $\phi\in C^{2s+\alpha}(\overline{B_{1/2}})$ and
		\begin{equation*}
			\norm{\phi}_{C^{2s+\alpha} (B_{1/2})} \le C_{d,s ,\alpha, \mu} \big( \norm{\phi}_{C^{\alpha}(\BR^d)}  + \norm{\psi}_{C^{\alpha}(B_1)} \big),
		\end{equation*}
		whenever $2s+\alpha $ is not an integer.
	\end{theorem}
		Now we combine the regularity results of Ros-Oton and Serra for solutions to \eqref{eq:fracequation} with \cref{prop:regularity} and the weak existence from \cref{th:pointwisesolution} to prove the following existence and uniqueness of classical solutions. 
	\begin{proposition}\label{prop:interiorreg}
		Let $\Omega$ be a bounded Lipschitz domain satisfying uniform exterior ball condition, $s\in (0,1)$ and $\mu$ a bounded measure on $S^{d-1}$ satisfying \eqref{eq:ellip_conditions}. We fix $\psi\in C_c^\infty(\Omega)$. There exists a classical solution $\phi\in C^{2s+\alpha}(\Omega)\cap C^s(\overline{\Omega})$ to
		\begin{align*}
			A_s \phi = \psi \text{ in }\Omega,\\
			\phi = 0 \text{ on }\Omega^c. 
		\end{align*}
	\end{proposition}
	\begin{proof}
		By \cref{th:pointwisesolution}, there exists a bounded weak solution $\phi\in L^\infty(\Omega)$ to  \eqref{eq:fracequation} in the sense of \cref{def:weaksol}. We know $\phi\in C^s(\BR^d)\cap C^{2s+\alpha}(\Omega)$ by \cref{prop:regularity} and \cref{th:interior_regularity}. Here $\alpha>0$ such that $2s+\alpha$ is not an integer. By \cref{prop:double_difference}, $A_s \phi$ is continuous in $\Omega$. We fix $\eta\in C_c^\infty(\Omega)$. Since $A_s\phi$ is continuous in $\Omega$, $A_s\phi \in L^\infty(\supp(\eta))$. The choice of $\phi$, see \eqref{eq:weak_sol}, yields
		\begin{equation*}
			\int\limits_{\Omega} \psi(x)\eta(x)\di x = \CE_{A_s}(\phi, \eta).
		\end{equation*}
		$A_s\phi \in L^\infty(\supp (\eta))$, $\eta\in C_c^\infty(\Omega)$, \eqref{eq:greengauß_kappa} and dominated convergence yield
		\begin{equation*}
			\int\limits_{\Omega} A_s\phi(x) \eta(x)\di x= \CE_{A_s}(\phi, \eta)= \int\limits_{\Omega} \psi(x)\eta(x)\di x.
		\end{equation*}
		This identity holds for all $\eta\in C_c^\infty(\Omega)$. Since $A_s\phi$ is continuous in $\Omega$, we conclude $A_s \phi(x)= \psi(x)$ for all $x\in \Omega$. 
	\end{proof}
	The following lemma shows how $A_s u$ changes under mollification of $u$. This result is standard and well known for translation invariant operators, \eg see \cite[Theorem 3.12]{bogdan_tail} for the fractional Laplacian. We prove the statement for the convenience of the reader.
	\begin{lemma}\label{lem:mpollification}
		Let $u\in L^1(\Omega)\cap L^1(\BR^d, \nu^\star(x)\dishort x)$ satisfy \eqref{eq:ultrasubharmonic} and \eqref{eq:nonnegativeoutside}. Fix a radial bump function $\eta \in C_c^\infty(\BR^d)$ with $\supp \eta = \overline{B_1(0)}$, $\eta \ge 0 $ and $\norm{\eta}_{L^1(B_1(0))}=1$. We define an approximate identity $\eta_{\varepsilon}:= \varepsilon^{-d}\eta(\frac{\cdot}{\varepsilon})$. The mollification $u_\varepsilon= u * \eta_{\varepsilon}$ satisfies
		\begin{align*}
			A_s u_\varepsilon \le 0 \text{ in } \Omega_{\varepsilon},\\
			u_\varepsilon \le 0 \text{ on } \Omega^\varepsilon
		\end{align*}
		pointwise. Here $\Omega_{\varepsilon}$ and $\Omega^\varepsilon$ are as in \eqref{eq:definition_omega_epsilon}.
	\end{lemma}
	\begin{proof} The second claim follows from \eqref{eq:nonnegativeoutside} and $\supp(\eta_\varepsilon)= \overline{B_\varepsilon(0)}$.\\
		 $u \in L^1(\Omega) \cap L^1(\BR^d, \nu^\star(x)\dishort x)$, $\eta_\varepsilon\in C_c^\infty(\BR^d)$ yields $u_\varepsilon\in C^\infty(\BR^d)\cap L^1(\BR^d, \nu^\star(x)\dishort x)$. By \cref{prop:double_difference}, $A_s u_\varepsilon$ is continuous in $\Omega_\varepsilon$. For the first property we fix $x\in \Omega_{\varepsilon}$. Now we use Fubini's theorem, the radiality of $\eta$ and the transformation theorem to conclude
		\begin{align*}
			A_s u_\varepsilon(x)&= \text{p.v.} \int\limits_{\BR} \int\limits_{S^{d-1}} \int\limits_{\BR^d} \frac{u(y)\eta_{\varepsilon}(x-y)-u(y) \eta_{\varepsilon}(x+r\theta-y)}{\abs{r}^{1+2s}} \di y \mu(\dishort\theta) \di r \\
			&= \int\limits_{\BR^d} u(y) \ \text{p.v.}  \int\limits_{\BR} \int\limits_{S^{d-1}} \frac{\eta_{\varepsilon}(x-y)- \eta_{\varepsilon}(x+r\theta-y)}{\abs{r}^{1+2s}}  \mu(\dishort\theta) \di r\di y\\
			&= \int\limits_{\BR^d} u(y) \  \text{p.v.} \int\limits_{\BR}\int\limits_{S^{d-1}}  \frac{\eta_{\varepsilon}(y-x)- \eta_{\varepsilon}(y+r\theta -x)}{\abs{r}^{1+2s}} \mu(\dishort\theta) \di r \di y\\
			&= \int\limits_{\BR^d} u(y) A_s(\eta_{\varepsilon} (\cdot -x))(y) \di y\le 0.
		\end{align*}
		The second equality is true by  \eqref{eq:fraclaphoelderinequality} and dominated convergence. The last inequality is true because $\eta_\varepsilon(\cdot -x) \in C_c^\infty(\Omega)$ for any $x\in \Omega_{\varepsilon }$.
	\end{proof}
	 The next lemma show the following. If $u$ is a subsolution, then $u^+=\max\{u,0\}$ is again a subsolution. Silvestre proved this result for the fractional Laplacian in \cite[Lemma 2.18]{Silvestre2007}. We follow the technique in \cite[Lemma 2]{mollificationidea_lemma1} by Li, Wu and Xu as it generalizes easier to our class of operators.
	\begin{lemma}\label{lem:uplus}
		If $u\in L^1(\Omega)\cap L^1(\BR^d, \nu^\star(x)\dishort x)$ satisfies \eqref{eq:ultrasubharmonic} and \eqref{eq:nonnegativeoutside}, then $u^+=\max \{  u,0\}$ satisfies \eqref{eq:ultrasubharmonic} and $u^+=0$ on $\Omega^c$.
	\end{lemma}
	\begin{proof}
		The proof of the second claim is immediate. We divide the proof into two steps. 
		
		\textbf{Step 1. } Suppose $u\in L^1(\Omega)\cap L^1(\BR^d, \nu^\star(x)\dishort x)\cap C^\infty(\BR^d)$ satisfies \eqref{eq:ultrasubharmonic}. Let $\eta\in C_c^\infty(\Omega)$. \eqref{eq:greengauß_kappa} and dominated convergence yield 
		\begin{equation*}
			\int\limits_{\Omega} A_s u(x) \eta(x)\di x = \CE_{A_s}(u,\eta)= \int\limits_{\BR^d} u A_s\eta(x)\di x.
		\end{equation*}	
		The application of dominated convergence is justified. By \cref{prop:double_difference}, $A_s u$ is continuous in $\Omega$ and, thus, $A_s u\in L^\infty(\supp \eta)$. Therefore, the first term $\int_{\Omega} A_s u(x) \eta(x)\di x$ exists. $\CE_{A_s}(u,\eta)$ exists by \cref{lem:energy_exists}. By \cref{prop:innerprod_exists} and $u\in L^1(\Omega)\cap L^1(\BR^d, \nu^\star(x)\dishort x)$, the last term $\int_{\BR^d} u A_s\eta(x)\di x$ exists. Since $\eta\in C_c^\infty(\Omega)$ was arbitrary, we conclude $A_s u(x) \le 0$ for all $x\in\Omega$.
		
		Now consider $P=\{ x\in \BR^d\,|\, u(x)>0 \}$. For $\kappa>0$ and $\eta \in C_c^\infty(\Omega)$ with $\eta\ge 0$ we define $A_s^{\kappa}\eta$ as in \eqref{eq:approx_operator}. 
		\begin{align*}
			\int\limits_{\BR^d} u^{+}(x) \, A^\kappa_s\eta(x) \di x&= \int\limits_{P} u(x) \, \int\limits_{\{|r|\ge \kappa\}} \int\limits_{S^{d-1}} \frac{\eta(x)-\eta(x+r\theta)}{\abs{r}^{1+2s}  }  \mu(\dishort\theta) \di r \di x\nonumber\\
			&= \int\limits_{P} \eta(x) \, \int\limits_{\{|r|\ge \kappa\}} \int\limits_{S^{d-1}} \frac{u(x)-u(x+r\theta)}{\abs{r}^{1+2s}  } \mu(\dishort\theta) \di r \di x\nonumber\\
			&\quad\quad+ \int\limits_{P} \int\limits_{\{|r|\ge \kappa\}} \int\limits_{S^{d-1}} \1_{P}(x+r\theta) \frac{u(x+r\theta)\eta(x)-u(x)\eta(x+r\theta)}{\abs{r}^{1+2s}  }\mu(\dishort\theta) \di r \di x\nonumber\\
			&\quad\quad+ \int\limits_{P}\int\limits_{\{|r|\ge \kappa\}} \int\limits_{S^{d-1}} \1_{P^c}(x+r\theta) \frac{u(x+r\theta)\eta(x)-u(x)\eta(x+r\theta)}{\abs{r}^{1+2s}  }  \mu(\dishort\theta) \di r\di x\nonumber\\
			&\quad\quad=: (I)+(II)+(III).\label{eq:step_1_u_plus}
		\end{align*}
		  These integrals exist since $\eta$ has compact support and $u\in L^1(\BR^d, \nu^\star(x)\dishort x)$. By symmetry, $(II)$ is zero and $(III)$ is nonpositive by the definition of $P$. We conclude
		\begin{align*}
			\int\limits_{\BR^d} u^{+}(x) \, A^\kappa_s\eta(x) \di x\le \int\limits_{P} \eta(x) \, A^\kappa_s u(x) \di x.
		\end{align*}
		By \cref{prop:double_difference} and dominated convergence, we conclude
		\begin{align*}
			\int\limits_{\BR^d} u^{+}(x) \, A_s\eta(x) \di x \le 0.
		\end{align*}
	
		\textbf{Step 2. } Now suppose $u\in L^1(\Omega)\cap L^1(\BR^d, \nu^\star(x)\dishort x) $ satisfies \eqref{eq:ultrasubharmonic} and \eqref{eq:nonnegativeoutside}. By \cref{lem:mpollification} and step 1, 
		\begin{align*}
			\big( (u_\varepsilon)^+,A_s \phi \big)_{L^2(\BR^d)} &\le 0 \text{ for any nonnegative } \phi \in C_c^\infty (\Omega_{\varepsilon}), \\
			(u_\varepsilon)^+ &=0 \text{ on } (\Omega^{\varepsilon})^c.
		\end{align*}
		Notice $u^+=0$ on $\Omega^c$ by assumption \eqref{eq:nonnegativeoutside}. Now consider any $\phi \in C_c^\infty(\Omega)$ with $\phi \ge 0$. Fix $\varepsilon_0>0$ such that $\supp \phi \subset \Omega_{\varepsilon_0}$. Then for any $0<\varepsilon<\varepsilon_0$ we have $\big( (u_\varepsilon)^+,A_s \phi \big)_{L^2(\BR^d)} \le 0$. Since $\abs{a^+-b^+}\le \abs{a-b}$ holds for any two real numbers $a,b$, we notice for any $0<\varepsilon<\varepsilon_0$
		\begin{align*}
			\norm{(u_\varepsilon)^+- u^+}_{L^1(\Omega_{\varepsilon})}&=  \int\limits_{\Omega_{\varepsilon }} \abs{(u_\varepsilon)^+(x)- u^+(x)}\di x\le \int\limits_{\Omega_{\varepsilon}} \abs{u_\varepsilon(x)-u(x)}  \di x\\
			&= \int\limits_{\Omega_{\varepsilon}} \abs{(u\, \1_\Omega )*\eta_\varepsilon(x)-u(x)\, \1_\Omega(x) }  \di x\le \norm{(u\, \1_\Omega )*\eta_\varepsilon-u\, \1_\Omega}_{L^1(\BR^d)}.
		\end{align*}
		This term converges to zero for $\varepsilon \to 0$ since $u\,\1_\Omega\in L^1(\BR^d)$. Lastly,
		\begin{align*}
			\norm{(u_{\varepsilon})^+- u^+}_{L^1(\Omega^{\varepsilon}\setminus \Omega_{\varepsilon })}&=  \int\limits_{\Omega^{\varepsilon}\setminus \Omega_{\varepsilon }} \abs{(u_\varepsilon)^+(x)- u^+(x)}\di x\le  \int\limits_{\Omega^{\varepsilon}\setminus \Omega_{\varepsilon }} \abs{(u_\varepsilon)^+(x)}+ \abs{u^+(x)}\di x\\
			&\le \int\limits_{\Omega^{\varepsilon}\setminus \Omega_{\varepsilon }} \abs{(u^+)_\varepsilon(x)}+\abs{u^+(x)} \di x\le \int\limits_{\Omega^{\varepsilon}\setminus \Omega_{\varepsilon }} \abs{(u^+)_\varepsilon(x) - u^+(x)}+2\,\abs{u^+(x)} \di x\\
			&\le \norm{(u^+)_\varepsilon-u^+}_{L^1(\BR^d)} + 2\, \norm{u^+}_{L^1(\Omega^{\varepsilon } \setminus \Omega_{\varepsilon })}.
		\end{align*}
		 Because $u\in L^1(\Omega)$ and by assumption \eqref{eq:nonnegativeoutside} $u^+\in L^1(\BR^d)$. Therefore, the right-hand side of the previous inequality converges to zero as $\varepsilon\to 0$. Since $\phi\in C_c^\infty(\Omega)$, we know $\norm{A_s \phi}_{L^\infty(\Omega^{\varepsilon_0})}<\infty$ by \eqref{eq:fraclaphoelderinequality}. Thus, we conclude by Hölder's inequality 
		 $$
		 	\abs{ ((u_\varepsilon)^+-u^+,A_s \phi)_{L^{2}(\BR^d)} }\le \norm{(u_\varepsilon)^+-u^+}_{L^1(\Omega^{\varepsilon_0 })}\, \norm{A_s \phi}_{L^\infty(\Omega^{\varepsilon_0})} \to 0
		 $$
		 for $\varepsilon\to 0$ and $(u^+,A_s \phi)_{L^{2}(\BR^d)}\le 0$.  
	\end{proof}
The following lemma is the key technical estimate in the proof of \cref{Theorem 1.1}. It connects the regularity of weak solutions, see \cref{prop:regularity}, with the assumption \eqref{eq:boundary_regularity_assumption}. 
		\begin{lemma}\label{prop:calculation_distancefunction}
		Let $s\in(0,1)$ and $\Omega\subset \BR^d$ be a bounded Lipschitz domain satisfying uniform exterior ball condition. Let $\{D_\varepsilon\}_{\varepsilon>0}$ be $C^\infty$-domains exhausting $\Omega$ from \cref{kor:approximation_radius} with $\lambda>1$.  Additionally, we fix a nonnegative, $L^1$-normalized bump function $\eta\in C_c^\infty(\BR^d)$ with $\supp \eta = \overline{B_1(0)}$ and an approximate identity $\eta_\varepsilon:= \varepsilon^{-d}\, \eta(\tfrac{\cdot}{\varepsilon})\in C_c^\infty(\BR^d)$. There exists a constant $C>0$ such that
		\begin{equation*}
			\sup\{  \big( \dist{\,\cdot\,}{\partial D_{\varepsilon }}^{-s} \ast \eta_{\varepsilon} \big)(x)\,|\, x\in \Omega, \, \dist{x}{\partial \Omega}<(1+\lambda)\varepsilon  \}\le  C\,\varepsilon^{-s}
		\end{equation*}
		for all $0<\varepsilon<\varepsilon_0$.
	\end{lemma}
		\begin{proof}
		Fix any $0<\varepsilon<\varepsilon_0$ and any $x\in \Omega$ satisfying $\dist{x}{\partial \Omega}<(1+\lambda)\varepsilon$. Now we estimate the convoluted distance function.
		\begin{equation}\label{eq:proposition_distance_coarea_1}
			\dist{\,\cdot\,}{\partial D_{\varepsilon }}^{-s} \ast \eta_{\varepsilon} (x)= \int\limits_{\BR^d} \dist{y}{\partial D_{\varepsilon}}^{-s} \eta_{\varepsilon}(x-y) \di y \le \varepsilon^{-d}\, \norm{\eta}_{L^\infty}\, \int\limits_{B_{\varepsilon}(x)} \dist{y}{\partial D_{\varepsilon}}^{-s} \di y.
		\end{equation} 
		For any $y\in B_\varepsilon(x)$
		\begin{equation*}
			\dist{y}{\partial D_\varepsilon}\le \abs{y-x}+ \dist{x}{\partial \Omega}+ \sup\{ \dist{z}{\partial \Omega}\,|\, z\in \partial D_\varepsilon \}\le 2(1+\lambda)\varepsilon.
		\end{equation*}
		Now we apply coarea formula, see Federer \cite{coarea}, to the right-hand side of \eqref{eq:proposition_distance_coarea_1} with the Lipschitz continuous function $\dist{\cdot}{\partial D_\varepsilon}$, which satisfies $\abs{\nabla \dist{\cdot}{\partial D_\varepsilon}}=1$ \aev.
		\begin{equation}\label{eq:proposition_distance_coarea}
			\dist{\,\cdot\,}{\partial D_{\varepsilon }}^{-s} * \eta_{\varepsilon} (x)\le \varepsilon^{-d}\, \norm{\eta}_{L^\infty}\, \int_{0}^{2(1+\lambda)\varepsilon} t^{-s}\, \SH^{(d-1)} \big(  \{  y\in B_{\varepsilon}( x)\,|\, \dist{ y }{\partial D_{\varepsilon}} =t \} \big) \di t.
		\end{equation}
		Here $\SH^{(d-1)}$ is the $(d-1)$-dimensional Hausdorff measure. The Hausdorff measure of a ball $B_r$ intersecting a hyperplane scales like $r^{d-1}$. Thus, there exists a constant $c>0$ such that
		\begin{equation*}
			\SH^{d-1} \big(  \{  y\in B_{\varepsilon}( x)\,|\, \dist{ y }{\partial D_{\varepsilon}} =t \} \big)\le c\varepsilon^{d-1},
		\end{equation*}
		see \eg Ros-Oton and Valdinoci \cite[Lemma A.4 (A.19)]{RosOton_Vladinoci_appendix}. The constant $c$ does not depend on $\varepsilon$. Now we combine this estimate and \eqref{eq:proposition_distance_coarea} to conclude the claim.
		\begin{equation*}
			\dist{\,\cdot\,}{\partial \Omega_{\varepsilon }}^{-s} * \eta_{\varepsilon} (x)\le  c\,  \norm{\eta}_{L^\infty}\, \varepsilon^{-1}\, \int_{0}^{2(1+\lambda)\varepsilon} t^{-s}\di t=  2^{1-s}\, c\,  \norm{\eta}_{L^\infty}\, \frac{(1+\lambda)^{1-s}}{(1-s)} \varepsilon^{-s}.
		\end{equation*}
	\end{proof}
	\section{Proof of maximum principles}\label{sec:proof_theorem}
	In this section we provide the proofs of \cref{Theorem 1.1} and \cref{Corollary 1.4}.
	
	\textbf{Proof of \cref{Theorem 1.1}: }Let $u, \mu, \Omega$ be as in \cref{Theorem 1.1}, $s\in(0,1)$. It is sufficient to prove the result for $u^+$ by \cref{lem:uplus}. We fix a nonnegative, radial bump function $\eta\in C_c^\infty(\BR^d)$ with $\supp \eta =\overline{B_1(0)}$ as well as $\norm{\eta}_{L^1(\BR^d)}=1$ and construct an approximate identity by $\eta_\varepsilon := \varepsilon^{-d}\eta(\varepsilon^{-1} \, \cdot)$. Additionally, we fix the sequence of $C^\infty$ subdomains $\{ D_\varepsilon \}_{0<\varepsilon<\varepsilon_0}$ from \cref{kor:approximation_radius} with $\lambda>1$ as given therein. Recall the definition of $\Omega_{\varepsilon}, \Omega^\varepsilon$ from \eqref{eq:definition_omega_epsilon}. Notice that $\Omega_{\lambda \varepsilon}\subset D_\varepsilon \subset \Omega_\varepsilon$. Now fix an arbitrary nonnegative $\psi\in C_c^\infty(\Omega)$ and $0<\varepsilon_1<\varepsilon_0$ such that $ \supp \psi \ssubset \Omega_{\lambda\varepsilon_1}$. 
		
	For any $0<\varepsilon<\varepsilon_1$ we define the function $v_\varepsilon= u^{+} \ast\eta_\varepsilon$. \cref{lem:mpollification} and \cref{lem:uplus} yield
		\begin{align}
			A_s v_\varepsilon&\le 0\text{ in }  \Omega_{\varepsilon},\label{eq:proofmollification}\\
			v_\varepsilon&= 0 \text{ on } (\Omega^{\varepsilon})^c. \nonumber
		\end{align}
	Since $u^+\in L^1(\BR^d)$, we know that $v_\varepsilon \in C_c^\infty(\BR^d)$. For any $0<\varepsilon<\varepsilon_1$ let $\phi_{\varepsilon}\in C^s(\BR^d) \cap C^{2s+\alpha}(D_{\varepsilon})$ be the pointwise solution to 
	\begin{align*}
		A_s \phi_{\varepsilon} &= \psi \text{ in } D_{\varepsilon},\\
		\phi_{\varepsilon} &= 0 \text{ on } D_{\varepsilon}^c
	\end{align*}
	from \cref{prop:interiorreg}. Here $0<\alpha\le s$ is such that $2s+\alpha$ is not an integer. 
	
	Now we define $A_s^\kappa$ for any $\kappa>0$ as in \eqref{eq:approx_operator}. For any $0<\varepsilon<\varepsilon_1$ and $\kappa>0$ 
		\begin{align}
			&\int\limits_{D_{\varepsilon}} A_s^{\kappa} v_\varepsilon(x) \, \phi_{\varepsilon}(x)-  v_\varepsilon(x) \, A_s^{\kappa} \phi_{\varepsilon}(x) \di x \nonumber \\
			&\qquad = \int\limits_{D_{\varepsilon}} \int\limits_{S^{d-1}} \int\limits_{\{ \abs{r}\ge \kappa \}} \frac{v_\varepsilon(x) \, \phi_{ \varepsilon}(x+r\theta)-  v_\varepsilon(x+r\theta) \,  \phi_{ \varepsilon}(x)}{\abs{r}^{1+2s}} \di r \mu(\dishort\theta) \di x\nonumber\\
			& \qquad = \int\limits_{D_{\varepsilon}} \int\limits_{S^{d-1}} \int\limits_{\{ \abs{r}\ge \kappa \}} \1_{D_{\varepsilon}}(x+r\theta )\frac{v_\varepsilon(x) \, \phi_{ \varepsilon}(x+r\theta)-  v_\varepsilon(x+r\theta) \,  \phi_{ \varepsilon}(x)}{\abs{r}^{1+2s}} \di r \mu(\dishort\theta) \di x\nonumber\\
			&\qquad \quad-\int\limits_{D_{\varepsilon}} \int\limits_{S^{d-1}} \int\limits_{\{ \abs{r}\ge \kappa \}}\1_{\Omega^{\varepsilon}\setminus D_{\varepsilon}}(x+r\theta ) \frac{ v_\varepsilon(x+r\theta) \,  \phi_{ \varepsilon}(x)}{\abs{r}^{1+2s}} \di r \mu(\dishort\theta) \di x=: R_{\varepsilon,\kappa}.\label{eq:greengauss}
		\end{align}
		The first term in $R_{\varepsilon, \kappa}$ is zero by symmetry. Our goal is to show that the remainder $R_{\varepsilon,\kappa}$ converges to zero as $\varepsilon\to 0+$ uniformly in $\kappa$. An application of the transformation theorem yields
		\begin{align}\label{eq:remainder}
			\abs{R_{\epsilon,\kappa}}\le \int\limits_{\Omega^{\varepsilon}\setminus D_\varepsilon}v_\varepsilon(x) \int\limits_{S^{d-1}} \int\limits_{\{ \abs{r}\ge \kappa\}}\frac{  \phi_{ \varepsilon}(x+r\theta)}{\abs{r}^{1+2s}} \di r \mu(\dishort\theta) \di x.
		\end{align}
		For any $x\in \Omega^{\varepsilon} \setminus D_\varepsilon$ we know $\phi_\varepsilon(x)=0$. Therefore, \cref{prop:regularity} yields
		\begin{align*}
		 \int\limits_{S^{d-1}} \int\limits_{ \{ \abs{r}>\dist{x}{\partial D_{\varepsilon}} \vee \kappa  \} } \frac{\phi_{ \varepsilon}(x+r\theta)-0}{\abs{r}^{1+2s}} \di r \mu(\dishort\theta)&\le   \norm{\phi_{ \varepsilon}}_{C^{s}(\BR^d)} \,\int\limits_{S^{d-1}} \int\limits_{ \{ \abs{r}>\dist{x}{\partial D_{\varepsilon}}\vee \kappa \} } \frac{1 }{\abs{r}^{1+s}} \di r \di \mu (\theta) \\
		&\le C_{d, D_{\varepsilon},  s, \mu}\, \mu(S^{d-1})\, \norm{\psi}_{L^\infty(\BR^d)}  \dist{x}{\partial D_{\varepsilon}}^{-s}.
		\end{align*}
		There exists a constant $C_1>0$ such that $ C_{d, D_{\varepsilon},  s, \mu}\, \mu(S^{d-1})\le C_1$ by \cref{rem:constant}. This constant $C_1$ is independent of $\varepsilon$ because the exterior ball radius of $D_{\varepsilon }$ is independent of $\varepsilon$, see \cref{kor:approximation_radius}. Notice that 
		\begin{equation*}
			\dist{y}{\partial \Omega}\le \abs{x-y}+ \dist{x}{\partial \Omega}\le (1+\lambda)\varepsilon
		\end{equation*}
		 for any $x\in\Omega^{\varepsilon}\setminus D_\varepsilon$ and $y\in B_\varepsilon(x)\cap \Omega$. The previous calculation and \eqref{eq:remainder} yields the following bound on the remainder $R_{\varepsilon, \kappa}$.
		\begin{align}
			\abs{R_{\varepsilon, \kappa}}&\le C_1\, \norm{\psi}_{L^\infty(\BR^d)} \, \int\limits_{\Omega^{\varepsilon}\setminus D_\varepsilon}v_\varepsilon(x)  \dist{x}{\partial D_{\varepsilon}}^{-s} \di x   \nonumber \\
			&= C_1\, \norm{\psi}_{L^\infty(\BR^d)} \, \int\limits_{\Omega\setminus \Omega_{(1+\lambda)\,\varepsilon}} u^+(y) \int\limits_{\Omega^{\varepsilon}\setminus D_\varepsilon}\eta_\varepsilon(x-y)  \dist{x}{\partial D_{\varepsilon}}^{-s} \di x \di y\nonumber \\
			&= C_1\, \norm{\psi}_{L^\infty(\BR^d)} \, \int\limits_{\Omega\setminus \Omega_{(1+\lambda)\,\varepsilon}} u^+(y)\, \big(\dist{\cdot}{\partial D_{\varepsilon}}^{-s} \ast \eta_\varepsilon \big) (y) \di y.\nonumber
		\end{align}
		By the previous estimate and \cref{prop:calculation_distancefunction}, there exists a constant $C>0$ such that
		\begin{align}
			\abs{R_{\varepsilon, \kappa}}&\le  C\,C_1 \norm{\psi}_{L^\infty(\BR^d)} \,((1+\lambda)\varepsilon)^{-s} \int\limits_{\Omega\setminus \Omega_{(1+\lambda)\,\varepsilon} }u^+(x)\di x .\label{eq:case1_remainder}
		\end{align}
		This converges to zero as $\varepsilon\to 0+$ by assumption \eqref{eq:boundary_regularity_assumption}. 
		
		Now we finish the proof. \eqref{eq:proofmollification}, \eqref{eq:greengauss} and the choice of $\phi_{\varepsilon}$ yield
		\begin{align}
			0&\ge \int\limits_{D_{\varepsilon}} A_s v_\varepsilon(x) \, \phi_{ \varepsilon}(x)\di x =\lim\limits_{ \kappa \to 0+}\int\limits_{D_{\varepsilon}} A_s^\kappa v_\varepsilon(x) \, \phi_{ \varepsilon}(x)\di x \nonumber\\
			&=\lim\limits_{ \kappa \to 0+}\Big[\int\limits_{D_{\varepsilon}}  v_\varepsilon(x) \,A_s^\kappa \phi_{ \varepsilon}(x)\di x -\int\limits_{\Omega^{\varepsilon}\setminus D_\varepsilon}v_\varepsilon(x) \int\limits_{S^{d-1}} \int\limits_{\{ \abs{r}\ge \kappa \}}\frac{  \phi_{ \varepsilon}(x+r\theta)}{\abs{r}^{1+2s}} \di r \mu(\dishort\theta) \di x\Big]\nonumber\\
			&= \int\limits_{D_\varepsilon} (u^{+}*\eta_\varepsilon)(x) \, \psi(x) \di x+\lim\limits_{ \kappa \to 0+}R_{\varepsilon,\kappa}. \label{eq:finalcalcuation}
		\end{align}
		In the previous calculation we used dominated convergence and \cref{prop:double_difference}. \eqref{eq:case1_remainder} implies that $R_{\varepsilon,\kappa}$ converges to zero as $\varepsilon\to 0+$ uniformly in $\kappa$. We take the limit $\varepsilon\to 0+$ in \eqref{eq:finalcalcuation}. Thereby,
		\begin{equation*}
			0\ge \int\limits_\Omega u^{+}(x) \psi(x) \di x
		\end{equation*}
		because $u^+\in L^1(\BR^d)$, $\{\eta_{\varepsilon}\}_{\varepsilon}$ is an approximate identity and $\supp \psi \ssubset \Omega_{\lambda\varepsilon_1}\subset D_{\varepsilon_1}\subset D_{\varepsilon}\subset \Omega$. We conclude the claim as $\psi\in C_c^\infty(\Omega)$ nonnegative was chosen arbitrarily. \qed
		
	\textbf{Proof of \cref{Corollary 1.4}: } We want to apply \cref{Theorem 1.1}. Let $u\in H^s(\Omega)\cap L^1(\BR^d, \nu^\star(x)\dishort x)$ satisfy \eqref{eq:weakmaximum_cond1} and \eqref{eq:weakmaximum_cond2}. Fix any nonnegative $\eta\in C_c^\infty(\Omega)$. \eqref{eq:greengauß_kappa} and dominated convergence yield 
	\begin{equation*}
		\CE_{A_s}(u,\eta)= \int\limits_{\BR^d} u A_s\eta(x)\di x.
	\end{equation*}	
	The bilinear form $\CE_{A_s}(u,\eta)$ exists by \cref{lem:energy_exists}. By \cref{prop:innerprod_exists} and $u\in L^1(\Omega)\cap L^1(\BR^d, \nu^\star(x)\dishort x)$, the last term $\int_{\BR^d} u A_s\eta(x)\di x$ exists. Thus, the application of dominated convergence was justified. Therefore, \eqref{eq:ultrasubharmonic} is satisfied. Since $C_c^\infty(\Omega)$ is dense in $H^{s/2}(\Omega)$, Hölder inequality and the fractional Hardy inequality, see Chen and Song \cite[Corollary 2.4]{chensong_hardy} or Dyda \cite[Equation (17)]{dyda_hardy}, imply
	\begin{align*}
		\int\limits_{\Omega} \frac{\abs{u(x)}}{\dist{x}{\partial\Omega}^s}\di x \le \norm{\dist{\cdot}{\partial \Omega}^{-s/2}}_{L^2(\Omega)} \Big( \int\limits_{\Omega} \frac{u(x)^2}{\dist{x}{\partial\Omega}^s}\di x  \Big)^{1/2}\le C\, \norm{u}_{H^{s/2}(\Omega)}.
	\end{align*}
	An application of \cref{Theorem 1.1} finishes the proof, see \cref{rem:theorem}. \qed 
\appendix
	\section{ }\label{sec:appendix}
	The following lemma compares the tail space $L^1(\BR^d, (1+\abs{x})^{-d-2s}\dishort x)$ with $L^1(\BR^d, \nu^\star(x)\dishort x)$.
	\begin{lemma}\label{appendix:tailspace}
		Let $s\in (0,1)$, $\Omega\subset \BR^d$ open and bounded and $\nu$ be as in  \eqref{eq:levy_measure}. Suppose the measure $\mu$ on $S^{d-1}$ has a density with respect to the surface measure $\sigma$. Additionally, we assume that the density is bounded from below and above by positive constants. The corresponding tail space $L^1(\BR^d,\nu^\star(x)\dishort x)$, see \eqref{eq:nustar}, coincides with the space $L^1(\BR^d,(1+\abs{x})^{-d-2s}\dishort x)$.
	\end{lemma}
	Out of convenience we will use the notation $\simeq$, $\lesssim$ and respectively $\gtrsim$ and we say $f(x) \lesssim g(x)$ for real valued functions if there exists a constant $C>0$ such that $\forall x:\, f(x)\le C g(x)$. The definition of $\gtrsim$ is now self explanatory and $f(x)\simeq g(x)$ is defined as $ f(x)\lesssim g(x) \,\land\, f(x)\gtrsim g(x)$.
	\begin{proof}
		It is sufficient to show that the weights are comparable, \ie
		\begin{equation*}
			\nu^\star(x) \simeq \frac{1}{(1+\abs{x})^{d+2s}}.
		\end{equation*}
		We assume that $\mu\simeq \sigma$. Therefore
		\begin{align*}
		\nu^\star(x)&=\int\limits_{\BR} \int\limits_{S^{d-1}} \1_{\Omega}(x+r\theta) (1+\abs{r})^{-1-2s} \mu(\dishort\theta) \di r \simeq \int\limits_{\BR} \int\limits_{S^{d-1}} \1_{\Omega}(x+r\theta) (1+\abs{r})^{-1-2s}  \sigma(\dishort\theta) \di r \\
			&\simeq  \int\limits_{\BR^d} \1_{\Omega}(x+y) (1+\abs{y})^{-1-2s} \abs{y}^{1-d} \di y.
		\end{align*}
		We begin proving the upper bound on $\nu^\star$. We choose a ball $B_R(0)$ that contains $\Omega\subset B_R(0)$ and $\dist{\Omega}{\partial B_R(0)}= 2\,\diam{\Omega}$. We split the integration domain into $B_R(0)$ and its complement. Notice $(1+R)^{-1-2s}\le (1+\abs{y})^{-1-2s}\le  1$ for $y\in B_R(0)$. Thereby,
		\begin{equation*}
			\int\limits_{B_R(0)} \1_{\Omega}(x+y) (1+\abs{y})^{-1-2s} \abs{y}^{1-d} \di y \simeq 	\int\limits_{B_R(0)} \1_{\Omega}(x+y)  \abs{y}^{1-d} \di y.
		\end{equation*}
		By the choice of $R$ we know $y\in B_R(0), x+y\in \Omega$ implies $x\in B_{2R}(0)$. Therefore,
		\begin{equation*}
		\int\limits_{B_R(0)} \1_{\Omega}(x+y)  \abs{y}^{1-d} \di y \le \sigma(S^{d-1})R\, \1_{B_{2R}(0)}(x) \lesssim \frac{1}{(1+\abs{x})^{d+2s}} \1_{B_{2R}(0)}(x).
		\end{equation*}
		Now we prove an appropriate estimates for the integral over $B_R(0)^c$. Firstly notice 
		\begin{equation*}
			\abs{y}\ge \abs{x}-\abs{x+y}\ge \abs{x}-R\ge \tfrac{1}{2}\abs{x}
		\end{equation*}
		 for any $x\in B_{2R}(0)^c$ and $x+y\in \Omega\subset B_R(0)$. Therefore, for any $x\in B_{2R}(0)^c$
		 \begin{align}\label{eq:appendix_1}
		 	\int\limits_{B_R(0)^c} \1_{\Omega}(x+y) (1+\abs{y})^{-1-2s} \abs{y}^{1-d} \di y \le\abs{\Omega}\, \big(\tfrac{1}{2}\abs{x}\big)^{-d-2s}\1_{B_{2R}(0)^c}(x)\lesssim \frac{1}{(1+\abs{x})^{d+2s}}\1_{B_{2R}(0)^c}(x).
		 \end{align}
		 For any $x\in B_{2R}(0)$ 
		 \begin{align}\label{eq:appendix_2}
		 		\int\limits_{B_R(0)^c} \1_{\Omega}(x+y) (1+\abs{y})^{-1-2s} \abs{y}^{1-d} \di y\le \abs{\Omega}R^{1-d}\1_{B_{2R}(0)}(x)\lesssim \frac{1}{(1+\abs{x})^{d+2s}}\1_{B_{2R}(0)}(x).
		 \end{align}
	 	\eqref{eq:appendix_1} and \eqref{eq:appendix_2} yield the upper bound
		\begin{equation*}
			\nu^\star(x)\lesssim  \frac{1}{(1+\abs{x})^{d+2s}}.
		\end{equation*}
		Now we prove a lower bound on $\nu^\star$. For any $x,y\in \BR^d$ such that $x+y\in \Omega$ the estimate $\abs{y} \le \abs{x} + \max_{z\in \Omega}\abs{z}$ holds. Therefore
		\begin{align*}
			(1+\abs{y})^{-1-2s} &\ge (1+\abs{x} + \max_{z\in \Omega}\abs{z} )^{-1-2s} \gtrsim (1+\abs{x} )^{-1-2s},\\
			\abs{y}^{1-d}&\ge (\max_{z\in \Omega}\abs{z}+\abs{x})^{1-d}\gtrsim \abs{x}^{1-d} \ge (1+\abs{x})^{1-d}.
		\end{align*}
		Thus, for $x\in \BR^d$
		\begin{equation*}
			\int\limits_{\BR^d} \1_{\Omega}(x+y) (1+\abs{y})^{-1-2s} \abs{y}^{1-d} \di y\gtrsim (1+\abs{x} )^{-d-2s} \, 	\int\limits_{\BR^d} \1_{\Omega}(x+y) \di y \gtrsim \frac{1}{(1+\abs{x})^{d+2s}}.
		\end{equation*}
		We conclude the desired bounds
		\begin{equation*}
			\frac{1}{(1+\abs{x})^{d+2s}}\lesssim	\nu^{\star}(x) \lesssim \frac{1}{(1+\abs{x})^{d+2s}}.
		\end{equation*}
	\end{proof}
	Now we give an example of $\nu$, see \eqref{eq:levy_measure}, and a specific $\Omega$ such that functions in the tail space $L^1(\BR^d, \nu^\star(x)\dishort x)$ are not necessarily integrable on $\Omega$. 
	\begin{bsp}\label{appendix:counterexample}
		We consider the domain
		\begin{equation*}
			\Omega = \{(x,y) \in \BR^2 \,|\, 0<x,y<1, x<2y<3x \}\subset \BR^2,
		\end{equation*}
		$e_1= \icol{1\\0}$, $e_2 = \icol{0\\1}$ and the operator $(-\partial_{e_1}^2)^{s}+(-\partial_{e_2}^2)^{s}$, which corresponds to $\mu= \delta_{e_1}+\delta_{e_2}$ in \eqref{eq:levy_measure}. We define the function $u:\BR^2 \to \BR$, $u(x,y):= (x\, y)^{-1}$ for $(x,y)\in \Omega$ and $u(x,y)=0$ else. The singularity in the origin is not integrable and thus $u\notin L^1(\Omega)$. This also implies $u\notin L^1(\BR^d,(1+\abs{x})^{-d-2s}\dishort x)$. On the other hand recall the definition of the tail space $L^1(\BR^d, \nu^\star(x)\dishort x)$. In this case we estimate the decay weight by
		\begin{equation*}
			\nu^{\star}(x,y) \le \int\limits_{\BR}  \1_\Omega (x+r,y)  \di r +  \int\limits_{\BR}  \1_\Omega(x,y+r)\di r.
		\end{equation*}
		Notice for any $(x,y)\in \Omega$ 
		\begin{equation*}
			\int\limits_{\BR} \1_\Omega (x+r,y)    \di r \le \frac{4y}{3},\quad\int\limits_{\BR} \1_\Omega (x,y+r)    \di r \le x.
		\end{equation*}
		And by
		\begin{align*}
		\infty&> \frac{4}{3}\ln(3) + \ln(3) \ge \frac{4}{3}\int\limits_{\Omega}  x^{-1}  \di (x,y) + \int\limits_{\Omega}   y^{-1} \di (x,y)\\
		&\ge \int\limits_{\Omega}  (x\,y)^{-1} \int\limits_{\BR} \1_\Omega (x+r,y)  \di r \di (x,y) + \int\limits_{\Omega}   (x\,y)^{-1} \int\limits_{\BR} \1_\Omega (x,y+r)  \di r\di (x,y)\\
		&\ge \int\limits_{\BR^2} \abs{u(x,y)} \nu^\star(x,y) \di (x,y),
		\end{align*}
		we conclude $u\in L^1(\BR^d, \nu^\star(x)\dishort x)$.
	\end{bsp} 
\section{ }\label{appendix B}
In this section we prove the maximum principle for the Laplacian formulated in the introduction. The strategy of the proof is the same as for \cref{Theorem 1.1}.  We consider a sequence $D_\varepsilon \ssubset \Omega$ of subdomains exhausting $\Omega$ and construct an approximation of the Green function by solving the Dirichlet problem $-\Delta \phi_\varepsilon = \psi \in C_c^\infty(\Omega)$ in $D_\varepsilon$ and $\phi_\varepsilon = 0$ on $\partial D_\varepsilon$. Then we use regularity results for solutions.
\begin{theorem}\label{th:main_local}
	Let $\Omega$ be a bounded $C^{1,\gamma}$-domain, $\gamma\in(0,1)$ and $u\in L_{\text{loc}}^1(\Omega)$. If $u$ satisfies
	\begin{align}
		(u, -\lap \eta)_{L^2(\Omega)}&\le 0 \text{ for all } \eta \in C_c^\infty(\Omega), \eta\ge 0,\\
		\lim\limits_{\varepsilon\to 0+} \varepsilon^{-1}\mspace{-50mu}\int\limits_{\{ x\in \Omega\,|\, \dist{x}{\partial \Omega}<\varepsilon \}} \mspace{-55mu}&u^+(x)\di x = 0,\label{eq:boundary_assumption_local}
	\end{align}
	then $u\le0 $ \aev on $\Omega$.
\end{theorem}   
\begin{rem}\label{rem:1}
	\begin{enumerate}[(i)]
		\item{ Let $u\in L_{\text{loc}}^1(\Omega)$. \eqref{eq:weaksubharmonicity} is equivalent to $u$ has a modification which is upper-semicontinuous in $\Omega$ and satisfies sub-mean-value property, \ie $u(x)\le \tfrac{1}{\abs{\partial B(x)}}\int_{\partial B(x)} u(y)\di y$ for all balls $B(x)\ssubset \Omega$, see the classical book \cite[p. 128]{Donogue_book} by Donoghue. }
		\item{ If $u\in L_{\text{loc}}^1(\Omega)$ satisfies \eqref{eq:weaksubharmonicity} and attains its global maximum in $\Omega$, then $u$ is constant. }
		\item{ If $u\in L_{\text{loc}}^1(\Omega)$ satisfies \eqref{eq:weaksubharmonicity} and \begin{equation*}
				\limsup_{x\to z} u(x)\le 0 \text{ for all } z\in \partial \Omega,
			\end{equation*} then $u\le 0$ in $\Omega$.   }
		\item{ The regularity assumption on $\Omega$ in \cref{th:main_local} is needed because we use boundary regularity of solutions to the Poisson problem \eqref{eq:poisson_local}.}
		\item{ We may replace the condition \eqref{eq:boundary_assumption_local} with the slightly stronger but more accessible condition \begin{equation*}
				u^+\in L^1(\Omega, \dist{x}{\partial\Omega}^{-1} \di x ).
		\end{equation*} }
	\end{enumerate}
\end{rem}
\begin{proof}
	For any $\varepsilon>0$ we define $\Omega_{\varepsilon}:= \{ x\in \Omega\,|\, \dist{x}{\partial \Omega}>\varepsilon \}$. We fix a sequence of $C^{2,\gamma}$-subdomains $\{ D_\varepsilon \}_{0<\varepsilon<\varepsilon_0}$ such that  
	\begin{enumerate}[(a)]
		\item{ $D_\varepsilon \ssubset \Omega$, $D_\varepsilon\subset D_{\varepsilon'}$ for $\varepsilon>\varepsilon'$ and $\cup_{0<\varepsilon<\varepsilon_0}D_\varepsilon=\Omega$, }
		\item{ there exists $\lambda>1$ such $\Omega_{\lambda\varepsilon}\subset D_\varepsilon\subset \Omega_{\varepsilon}$, }
		\item{ the surfaces $\partial D_\varepsilon$ are uniformly in $C^{1,\gamma}$. }
	\end{enumerate}
	See for example the proof of \cite[Theorem 8.34]{gil_trud} by Gilbarg and Trudinger. 
	
	Let $\eta\in C_c^\infty(\BR^d)$ be a nonnegative, radial, $L^1$-normalized bump function with $\supp \eta = \overline{B_1(0)}$ and define $\eta_\varepsilon:= \varepsilon^{-d}\eta(\tfrac{\cdot}{\varepsilon})$. Clearly, $u^+=\max\{ u,0 \}$ is also upper-semicontinuous on $\Omega$ and satisfies sub-mean-value property. By \cref{rem:1}(i), $u^+$ satisfies \eqref{eq:weaksubharmonicity}. We define $v_\varepsilon:= u^+\ast \eta_\varepsilon$ on $\Omega_{\varepsilon}$, $v_\varepsilon\in C^\infty(\Omega_\varepsilon)$. Thereby, for any $x\in \Omega_\varepsilon$
	\begin{equation*}
		-\Delta v_\varepsilon(x)= \int_\Omega u^+(y) (-\Delta \eta(x-\cdot))(y) \di y\le 0.
	\end{equation*}
	Now, we fix $\psi \in C_c^\infty(\Omega)$ and $\varepsilon_0>\varepsilon_1>0$ such that $\supp\psi \subset D_{\varepsilon_1}$. By \cite[Theorem 6.14, 8.33]{gil_trud}, there exists a strong solution $\phi_\varepsilon\in C^{2,\gamma}(\overline{D_\varepsilon})$ to the problem
	\begin{align}
		-\Delta \phi_\varepsilon&= \psi \text{ in } D_\varepsilon, \nonumber\\
		\phi_\varepsilon&= 0 \text{ on }\partial D_\varepsilon.\label{eq:poisson_local}
	\end{align}
	Additionally, there exists a constant $C>0$ such that 
	\begin{equation}\label{eq:regularity_local}
		\norm{\phi_\varepsilon}_{C^{1,\gamma}(D_\varepsilon)} \le C  \norm{\eta}_{L^\infty} .
	\end{equation}
	By (a) and (c), the constant $C$ can be chosen independently of $\varepsilon$, see proof of \cite[Theorem 8.34]{gil_trud}. 
	
	By Green-Gauß formula, for any $0<\varepsilon<\varepsilon_1$
	\begin{align*}
		0&\ge \int\limits_{D_\varepsilon} (-\Delta v_\varepsilon(x))\phi_\varepsilon(x)\di x \\
		&= \int\limits_{D_\varepsilon} v_\varepsilon(x)\psi(x)\di x - \int\limits_{\partial D_\varepsilon} \partial_n v_\varepsilon(x) \, \phi_\varepsilon \sigma(\dishort x) + \int\limits_{\partial D_\varepsilon}  v_\varepsilon(x) \,\partial_n \phi_\varepsilon \sigma(\dishort x)\\
		&= (I)+(II)+(III).
	\end{align*}
	The term $(II)$ is zero by the choice of $\phi_\varepsilon$. Since $u^+\in L^1(\Omega)$, $(I)$ converges to $\int_{\Omega} u^+ \psi$ as $\varepsilon\to 0$. It remains to show that $(III)$ converges to zero as $\varepsilon\to 0$. 
	\begin{align*}
		\abs{(III)}&\le \int\limits_{\Omega\setminus \Omega_{(1+\lambda)\varepsilon}} u^+(y) \int\limits_{\partial D_\varepsilon} \eta_\varepsilon(x-y) \abs{\partial_n \phi_\varepsilon(x)} \sigma(\dishort x) \di y\\
		&\le \varepsilon^{-d}\, \norm{\eta}_{L^\infty}\,\int\limits_{\Omega\setminus \Omega_{(1+\lambda)\varepsilon}} u^+(y) \int\limits_{\partial D_\varepsilon \cap B_\varepsilon(y)} \abs{\partial_n \phi_\varepsilon(x)} \sigma(\dishort x) \di y.
	\end{align*}
	There exist a constant $C_1>0$ such that $\sigma(\partial D_\varepsilon \cap B_\varepsilon(y))\le C_1 \varepsilon^{d-1}$ for all $y\in \Omega\setminus \Omega_{(1+\lambda)\varepsilon}$ and $\varepsilon<\varepsilon_1$ by Ros-Oton and Valdinoci \cite[Lemma A.4 (A.19)]{RosOton2016}. We use this and \eqref{eq:regularity_local} to estimate $(III)$ further.
	\begin{align*}
		\abs{(III)}\le C\,C_1 \norm{\psi}_{L^\infty}\, \norm{\eta}_{L^\infty}\,\varepsilon^{-1}\,\int\limits_{\Omega\setminus \Omega_{(1+\lambda)\varepsilon}} u^+(y)  \di y.
	\end{align*}
	By the assumption \eqref{eq:boundary_assumption_local}, $(III)$ converges to zero as $\varepsilon\to 0+$. We conclude 
	\begin{equation*}
		0\ge \int\limits_\Omega u^+(x)\psi(x)\di x.
	\end{equation*}
	Since $\psi\in C_c^\infty(\Omega)$ was arbitrary, $u^+=0$ \aev in $\Omega$.
\end{proof}
	
	\printbibliography
	
\end{document}